\documentclass[11pt,english]{article}
\usepackage[utf8]{inputenc}
\usepackage{lmodern}

\usepackage[utf8]{inputenc}
\usepackage[a4paper]{geometry}
\usepackage{color}
\usepackage{mathrsfs}
\usepackage{mathtools}
\usepackage{bm}
\usepackage{amsmath}
\usepackage{amssymb}
\usepackage{amsthm}
\usepackage{hyperref}
\hypersetup{colorlinks=true, linkcolor=blue, citecolor=red}
\usepackage[english]{babel}
\usepackage{dsfont}
\usepackage{changepage}
\usepackage{kantlipsum}

\newcommand{\one}{\mathds{1}}
\newcommand{\X}{\mathbb{X}}

\newcommand{\R}{\mathbb{R}}

\newcommand{\mres}{\mathop{\hbox{\vrule height 7pt width .5pt depth 0pt
\vrule height .5pt width 6pt depth 0pt}}\nolimits}

\newcommand{\eps}{\varepsilon}

\def\d{\mathrm{d}}
\def\mod{\mathrm{Mod}}

\catcode`@=11 \@addtoreset{equation}{section} \catcode`@=12
\def\R{\mathbb{R}}
\def\N{\mathbb{N}}

\def\loc{\mathrm{loc}}
\def\one{\mathds{1}}

\def\lploc{L^p_{\loc}}

\def\lip{\mathrm{Lip}}
\def\omt{\Omega_\tau}
\def\leb{\mathscr{L}}

\def\ztau{(0,\tau)}

\def\Y{\mathbf{Y}}
\def\S{\mathcal{S}}

\theoremstyle{plain}
	\newtheorem{theorem}{Theorem}[section]

\theoremstyle{definition}
	\newtheorem{definition}[theorem]{Definition}
	
	\newtheorem{open problem}[theorem]{Open Problem}
	\newtheorem{remark}[theorem]{Remark}

\makeatletter
\newcommand{\subjclass}[2][1991]{
  \let\@oldtitle\@title
  \gdef\@title{\@oldtitle\footnotetext{#1 \emph{Mathematics subject classification.} #2}}
}
\newcommand{\keywords}[1]{
  \let\@@oldtitle\@title
  \gdef\@title{\@@oldtitle\footnotetext{\emph{Key words and phrases.} #1.}}
}
\makeatother

\allowdisplaybreaks

\title{Time-smoothing for parabolic variational\\problems in metric measure spaces}
\author{Vito Buffa\footnote{Bologna, Italy. E-mail: bff.vti@gmail.com. \textsf{ORCID iD}: 0000-0003-4175-4848.}}
\date{\today}
\begin{document}

\maketitle

\begin{abstract}\noindent 
In 2013, Masson and Siljander determined a method to prove that the $p$-minimal upper gradient $g_{f_\eps}$ for the time mollification $f_\eps$, $\eps>0$, of a parabolic Newton-Sobolev function $f\in L^p_\loc(0,\tau;N^{1,p}_\loc(\Omega))$, with $\tau>0$ and $\Omega$ open domain in a doubling metric measure space $(\X,d,\mu)$ supporting a weak $(1,p)$-Poincar\'e inequality, $p\in(1,\infty)$, is such that $g_{f-f_\eps}\to0$ as $\eps\to0$ in $L^p_\loc(\omt)$, $\omt$ being the parabolic cylinder $\omt\coloneqq\Omega\times\ztau$. Their approach involved the use of Cheeger's differential structure, and therefore exhibited some limitations; here, we shall see that the definition and the formal properties of the parabolic Sobolev spaces themselves allow to find a more direct method to show such convergence, which relies on $p$-weak upper gradients only and which is valid regardless of structural assumptions on the ambient space, also in the limiting case when $p=1$.
\end{abstract}

\medskip

\section{Introduction}

\medskip

This work addresses a notable issue related to time-smoothing for time-dependent variational problems in the abstract setting of metric measure spaces, that is, the convergence of the mollified minimizers to the original functions with respect to the topology of the related parabolic functional spaces, namely spaces of Sobolev space-valued $L^p$ functions. While this property is known to hold even in the most general situation of arbitrary Banach space-valued $L^p$ functions by the standard theory of mollifiers (see for instance 
\cite{hvnvw,kjf,li}), when the target Banach space is a Sobolev space defined in the metric setting it is not straightforward to find a direct proof based explicitly on the ``metric gradients'', namely the {\it minimal $p$-weak upper gradients}, given the lack of linearity of such objects and the lack of smoothness in the ambient structure.
Indeed, when working with parabolic variational problems, in order to establish existence, regularity or other types of results for the minimizers, one needs to find first suitable partial estimates for both the time-mollified solution $f_\eps$ and its gradient, and then recovers the desired result by letting $\eps\to0$ and therefore by exploiting the convergence properties of time-smoothing. So, in the particular case of a time-dependent problem on a metric measure space, it is necessary to have an explicit proof of the smoothing result for the minimal $p$-weak upper gradients $g_{f_\eps}$ at hand in order to make use of such partial estimates.

The first successful attempt at showing this smoothing property for minimal $p$-weak upper gradients on metric measure spaces was featured in \cite{ms}, where the authors established a H\"older regularity condition for the (quasi) minimizers of the variational problem related to the parabolic $p$-Laplace equation, $p>2$, in the setting of a doubling metric measure space $(\X,d,\mu)$ supporting a weak $(1,p)$-Poincar\'e inequality. In their work, the functional space under consideration is the parabolic Newton-Sobolev space $L^p(0,\tau;N^{1,p}(\Omega))$, or its local version $L^p_\loc(0,\tau;N^{1,p}_\loc(\Omega))$.
As it is usually the case with parabolic variational problems, however, such time-regularity alone is not enough to prove the regularity condition, namely a {\it De Giorgi type estimate}. To this aim, the idea is to mollify the function $f$ in time via a standard time mollifier $\eta_\eps(s)$, $\eps>0$, and then to find the aforementioned partial estimates for both $f_\eps$ and its gradient; then, the final claim is realized by taking the limit as $\eps\to0$. 

In the Euclidean context this technique easily gives the sought-after De Giorgi estimate, thanks to the linearity of the gradient that allows for both $f_\eps\to f$ in $L^p_\loc(\omt)$ and $\nabla f_\eps\to \nabla f$ in $L^p_\loc(\omt)$ to hold as $\eps\to0$. In the metric setting, as we already said, the operation of taking upper gradients is not linear instead, so it seems not possible to prove $g_{f-f_\eps}\to0$ in $L^p_\loc(\omt)$ as $\eps\to0$ by simply relying on the theory of upper gradients. Indeed, in \cite[Lemma 6.8]{ms} the issue was circumvented by invoking a metric version of Rademacher's Theorem proved by J. Cheeger \cite[Theorem 4.38]{ch} and then by using the $\mu$-almost everywhere comparability between Cheeger's derivative and the minimal $p$-weak upper gradient applied to the time mollification $f_\eps$.

It is worth mentioning that, besides regularity, time-smoothing plays a relevant role also in the proof of the existence of a unique minimizer to a certain parabolic functional; this is for instance the case of \cite{bdm}, where this technique is used extensively to prove existence for the Total Variation Flow (\textsc{tvf}) in $\R^n$. More recently, the use of time mollification has gathered increasing attention also in the non-smooth setting, and applications of \cite[Lemma 6.8]{ms} have led to interesting results related to existence \cite{bcp,cohe}, stability \cite{fh2}, higher integrability \cite{fh,hab,mmpp,mp}, comparison principles and Harnack inequalities \cite{kinmas,mm}.

All the above mentioned works (with the only exception of \cite{bcp}), however, deal with problems where $p>1$ on doubling spaces supporting a Poincar\'e inequality. In fact the result contained in \cite[Lemma 6.8]{ms}, despite its depth and strength, shows however the following limitations implied by the use of Cheeger's theory:
\begin{enumerate}
 \item\label{quest-p} The range of validity is that of all the exponents $p\in(1,\infty)$, excluding therefore the limiting case $p=1$ which corresponds, for instance, to time-dependent variational problems involving functions of bounded variation ($BV$), like the \textsc{tvf}.
 \item\label{quest-dp} The setting is limited only to metric measure spaces satisfying the doubling and Poincar\'e requirements.
\end{enumerate}

Regarding the first issue, we observe that one of the most widely used notions of $BV$ functions, in the metric setting,
makes use of a relaxation procedure performed on minimal 1-weak upper gradients of Lipschitz functions, \cite{mir}, with the total variation of a function $f$ on any domain $\Omega\subset\X$ being defined as
\begin{equation*}
 \|Df\|(\Omega)\coloneqq\inf\left\{\underset{j\to\infty}{\lim\inf}\int_\Omega g_{f_j}\d\mu;\;(f_j)_{j\in\N}\subset\lip_\loc(\Omega),\;f_j\underset{j\to\infty}{\longrightarrow} f\:\mathrm{in}\:L^1(\Omega)\right\}.
\end{equation*}

Therefore, it appears necessary to extend the validity of \cite[Lemma 6.8]{ms} to the limiting exponent $p=1$ in order to employ successfully the time-smoothing technique for the \textsc{tvf} problem if one aims at showing, for instance, some regularity properties for its minimizers. This was for example one of the key-issues in \cite{bkp}, where it was necessary to use explicitly the partial estimates made possible by the main result in the present work in order to establish a De Giorgi type estimate for the variational solutions of the \textsc{tvf}, and then their continuity at some point\footnote{The lack of partial estimates at the early stages of \cite{bkp} was actually the starting point of the present study.}.

It turns out that the solution to this now longstanding question is contained in the very definition of the parabolic spaces $L^p(0,\tau;N^{1,p}(\Omega))$ itself and in the related measure-theoretic properties of such spaces, combined with the formal notion of parabolic minimal $p$-weak upper gradients and their behavior and, last but not least, with the customary notion of time mollification of a parabolic Sobolev function via a standard time mollifier.

All of these technical tools do not require to impose any structural assumption on $(\X,d,\mu)$ - like doubling measures or Poincar\'e inequalities - and in particular, as just explained, they allow for the time-smoothing result to hold true even when $p=1$. Thus said, we stress the fact that our result, namely Theorem \ref{vbthm} below, is not just an improvement of what is already available in the literature but it is actually a new result, as it is proved through radically different arguments that rely just on basic definitions and properties which not only are independent of a specific structure of the underlying metric measure space and avoid to involve any extra machinery, but moreover they also provide a way to recover the desired convergence by just using the theory of weak upper gradients alone, giving therefore a positive answer to a longstanding open problem.

\bigskip

\medskip

The paper is organized as follows.

\begin{itemize}
 \item In Section \ref{sec-sobolev} we start with the standard definition of Newton-Sobolev spaces $N^{1,p}(\X)$, $p\in[1,\infty)$, by introducing the basic notions of $p$-Modulus of a family of curves and of $p$-weak upper gradients, recalling their salient properties.
 \item In Section \ref{sec-t-dep} we introduce the time-dependent case, namely the parabolic Newton-Sobolev spaces $L^p(0,\tau;N^{1,p}(\Omega))$, where $p\in[1,\infty)$, $\tau>0$ and $\Omega\subset\X$ is an open set. We start by recalling the essentials of the theory for Banach space-valued $L^p$ functions in Section \ref{sec-banach-lp} and then in Section \ref{sec-par-sob} we specialize our discussion for the parabolic Newton-Sobolev spaces by stressing their measure-theoretic properties in connection with the general theory, in particular for the parabolic minimal $p$-weak upper gradients.
 We then conclude with Section \ref{sec-smoothing} where we eventually consider the main problem of the present work by first recalling the time-smoothing technique and then by showing Theorem \ref{vbthm}.
\end{itemize}

\bigskip

\section{Sobolev Spaces}\label{sec-sobolev}

\medskip

In this work, $(\X,d,\mu)$ will always be a complete and separable metric measure space endowed with a non-negative Radon measure $\mu$.

The Lebesgue spaces $L^p(\X,\mu)$, as well as $L^p(\Omega,\mu)$ for any domain (open set) $\Omega\subset\X$, $p\in[1,\infty]$, will be defined in the usual way; see for instance  \cite[Section 3.2]{hkst}. 
Since we are going to work with the reference measure $\mu$ only, this will be omitted from the notation and we shall simply write $L^p(\X)$ or $L^p(\Omega)$.

By $L^p_\loc(\Omega)$ we shall intend the space of functions in $L^p(U)$ with $U\Subset\Omega$, where such inclusion has to be read as 
\begin{equation*}U\subset\Omega\qquad\text{is\;bounded\;and}\qquad\text{dist}(U,\Omega^c)>0.\end{equation*}
The same notation and interpretation will apply also to the spaces of local Sobolev functions.

In the present section we shall discuss the notion of Sobolev spaces by means of upper gradients. We will consider only the classical ``Newtonian'' characterization of the Sobolev spaces $N^{1,p}(\X)$ \cite{hkst,sh}, even though other equivalent definitions are available in the literature; we refer the interested reader to \cite{ags-calc,ags-dens}.

The characterizations of Sobolev functions on open sets $\Omega\subset\X$, as well as their local versions, will follow by simply considering the domain $\Omega$ as a metric space in its own respect, together with the restrictions $d_\Omega$ and $\mu\mres\Omega$ of the distance $d$ and of the measure $\mu$ to $\Omega$, respectively.

\subsection{Newton-Sobolev spaces}\label{new-sob}

\medskip

The ``Newtonian'' approach to first-order Sobolev spaces in metric measure spaces is perhaps the most classical and known in the literature. 
Based on the now-familiar concept of \textit{weak upper gradient}, this characterization is rooted in the seminal papers \cite{hk1,hk2}, where a notion of \textit{very weak gradient} made its first appearance (the denomination \textit{upper gradient} came up slightly afterwards,  \cite{km}),  and was later developed by \cite{sh} via an application of the theory of $p$-modulus for families of curves previously studied by \cite{fu}.

The very brief presentation we shall give here is loosely adapted from the discussion of \cite{hkst}.

\medskip

\begin{definition}\label{def-modulus}
 Let $\Gamma\subset AC([0,1],\X)$ denote a family of absolutely continuous curves $\gamma:[0,1]\to\X$. For $p\in[1,\infty)$, the $p$\textit{-modulus} of $\Gamma$ is defined as the quantity
 \begin{equation}\label{modulus}
  \mod_p(\Gamma)\coloneqq\inf\left\{\int_\X \rho^p\mathrm{d}\mu;\;\rho:\X\to[0,\infty]\;\text{Borel},\;\int_\gamma \rho\mathrm{d}s\ge1\:\forall\gamma\in\Gamma\right\}.
 \end{equation}
 Any map $\rho$ as in \eqref{modulus} will be called an \textit{admissible density}. We shall say that $\Gamma$ is $\mod_p$\textit{-negligible} whenever $\mod_p(\Gamma)=0$.
\end{definition}

\smallskip

\begin{definition}
 Given a function $f:\X\to\overline{\mathbb{R}}$, a Borel map $g:\X\to[0,\infty]$ is said to be an \textit{upper gradient} for $f$ if
 \begin{equation}\label{up-grad}
  \left\vert f(\gamma_1)-f(\gamma_0)\right\vert\le\int_\gamma g\mathrm{d}s
 \end{equation}
for every absolutely continuous curve $\gamma:[0,1]\to\X$. Here, $\gamma_0=\gamma(0)$ and $\gamma_1=\gamma(1)$ denote the endpoints of the curve $\gamma$, while $s$ stands for its arc-length parametrization.

If \eqref{up-grad} fails to hold on a $\mod_p$-negligible family, then $g$ will be called a $p$\textit{-weak upper gradient} for $f$.
\end{definition}

\smallskip

\begin{definition}
 We say that a $p$-integrable $p$-weak upper gradient $g$ of some function $f:\X\to\R$ is a \textit{minimal $p$-weak upper gradient} for $f$ whenever $g\le h$ $\mu$-almost everywhere for any $p$-weak $p$-integrable upper gradient $h$ of $f$; we shall denote it as $g_f$.
\end{definition}

\smallskip

Of course, when a minimal $p$-weak upper gradient exists, it is the one with smallest $L^p$ norm among $p$-weak upper gradients. In particular - see for instance \cite[Theorem 6.3.20]{hkst} - the minimal $p$-weak upper gradient is unique and any function which admits a $p$-integrable $p$-weak upper gradient has a minimal one.

\medskip

With these tools available, we can realize first the following notion of Sobolev-Dirichlet spaces:


\begin{definition}\label{dirn}
 The \textit{Newtonian Sobolev-Dirichlet class} $D^{1,p}(\X)$ consists of all the measurable maps $f:\X\to\R$ which possess a $p$-integrable $p$-weak upper gradient in $\X$. $D^{1,p}(\X)$ is a vector space equipped with the semi-norm
 \begin{equation*}
  \|f\|_{D^{1,p}(\X)}\coloneqq\|g_f\|_{L^p(\X)}.
 \end{equation*}
\end{definition}

\smallskip

\begin{remark}\label{up-grad-prop}We recall that the minimal $p$-weak upper gradient satisfies the following well-known properties:
\begin{enumerate}
 \item \underline{Sub-linearity}: $g_{\alpha u+\beta v}\le|\alpha|g_u+|\beta|g_v$ for all $u,v\in D^{1,p}(\X)$, $\alpha,\beta\in\R$;
 \item \underline{Weak Leibniz rule}: $g_{uv}\le|u|g_v+|v|g_u$ for all $u,v\in D^{1,p}\cap L^\infty(\X)$;
 \item \underline{Locality}: $g_u=g_v$ $\mu$-almost everywhere on $\{u=v\}$ for all $u,v\in D^{1,p}(\X)$;
 \item \underline{Chain rule}: if $\varphi:\R\to\R$ is Lipschitz and $u\in D^{1,p}(\X)$, then $\varphi\circ u\in D^{1,p}(\X)$ and $g_{\varphi\circ u}=|\varphi'\circ u|g_u$ $\mu$-almost everywhere.
\end{enumerate}

\end{remark}

\smallskip

\begin{definition}\label{sobn}
 Let us now consider the class of all $L^p(\X)$ functions which admit a $p$-integrable $p$-weak upper gradient, namely $\tilde{N}^{1,p}(\X)\coloneqq D^{1,p}\cap L^p(\X)$.
 
 On this vector space we define the semi-norm
 \begin{equation}\label{sobn-norm}
  \|f\|^p_{\tilde{N}^{1,p}(\X)}\coloneqq\|f\|^p_{L^p(\X)}+\|g_f\|^p_{L^p(\X)}.
 \end{equation}
 Then, the \textit{Newton-Sobolev space} $N^{1,p}(\X)$ is given as the normed space consisting of equivalence classes of functions in $\tilde{N}^{1,p}(\X)$, with any two functions $u,v$ being equivalent if and only if $\|u-v\|_{\tilde{N}^{1,p}(\X)}=0$. In other words,
 \begin{equation*}
  N^{1,p}(\X)\coloneqq\tilde{N}^{1,p}(\X)\Big/\left\{f\in\tilde{N}^{1,p}(\X);\;\|f\|_{\tilde{N}^{1,p}(\X)}=0\right\}.
 \end{equation*}
$N^{1,p}(\X)$ is a Banach space (see \cite[Theorem 7.3.6]{hkst}, or \cite[Theorem 3.7]{sh}) endowed with the quotient norm $\|\cdot\|_{N^{1,p}(\X)}$ defined in the same way as in \eqref{sobn-norm}.
\end{definition}

\medskip

\section{The time-dependent case}\label{sec-t-dep}

\medskip

In this section we present the construction of the time-dependent Sobolev spaces on space-time cylinders of the form $\Omega_\tau\coloneqq\Omega\times\ztau$, $\tau>0$, where $\Omega\subset\X$ is any open set.

Before giving the main definition, we recall the customary construction of Banach space-valued $L^p$ spaces, $p\in[1,\infty)$. As standard references for this, one might consider \cite{du,ds,hkst,ruz}.

\bigskip

\subsection{Generalities on Banach space-valued $L^p$ functions}\label{sec-banach-lp}

\medskip

Let $(\S,\Sigma,\nu)$ be a complete, $\sigma$-finite measure space, $\Y$ a Banach space and, for any given set $E$, let us denote by $\one_E$ its characteristic function,

\begin{equation*}
\one_E(x)\coloneqq \begin{cases} 1 & x\in E, \\ \mbox{} \\ 0 & x \notin E. \end{cases}
\end{equation*}

\begin{definition}
A function $f:A\to\Y$, $A\in\Sigma$, is said to be \textit{strongly measurable}, or just \textit{measurable}, if there is a sequence of simple functions $(f_k)_{k\in\N}\subset\Y$,
\begin{equation}\label{seq-simple}
 f_k=\sum_{i=1}^{n_k}\one_{E_i^{(k)}}\cdot v_i^{(k)}\quad\forall\,k\in\N,
\end{equation}
where $\{E_i^{(k)}\}_{i=1}^{n_k}$ is a measurable disjoint partition of $A$ and $\{v_i^{(k)}\}_{i=1}^{n_k}\subset\Y$, such that $f$ is the $\nu$-almost everywhere limit of the $f_k$'s, namely (upon relabeling)
\begin{equation}\label{ae-simple-abstract}
 f=\sum_{k=1}^\infty \one_{E_k}\cdot v_k
\end{equation}
with the $E_k$'s and the $v_k$'s to be intended as in \eqref{seq-simple}.
\end{definition}

Thus said, the characterization of $\Y$-valued $L^p$ functions reads as follows:

\smallskip

\begin{definition}\label{def-ban-lp}
 Let $A\in\Sigma$ and let $p\in[1,\infty)$. We define $L^p(A;\Y)$ as the vector space of all equivalence classes of measurable functions $f:A\to\Y$ such that
 \begin{equation}\label{leb-int-norm}
  \int_A\|f\|_{\Y}^p\d\nu<\infty,
 \end{equation}
which means that $\|f\|^p_\Y$ is $\nu$-integrable. Here, two functions are considered equivalent whenever they coincide $\nu$-almost everywhere.

$L^p(A;\Y)$ equipped with the norm
\begin{equation}\label{ban-lp-norm}
 \|f\|_{L^p(A;\Y)}\coloneqq\left(\int_A\|f\|^p_{\Y}\d\nu\right)^\frac{1}{p}
\end{equation}
is a Banach space.
\end{definition}

We observe that \eqref{ae-simple-abstract} is equivalent to ask that $\|f(s)-f_k(s)\|_\Y\to0$ as $k\to\infty$ for $\nu$-almost every $s$. Moreover, the measurability requirement combined with \eqref{leb-int-norm} entails the Bochner integrability of $f$ by Bochner's Theorem \cite{boch}, meaning that
\begin{equation*}
 \underset{k\to\infty}{\lim}\int_A\|f-f_k\|_\Y\d\nu\longrightarrow0.
\end{equation*}


\subsection{Parabolic Sobolev Spaces}\label{sec-par-sob}

\medskip

Based on the preceding section, we can now characterize the class of parabolic Newton-Sobolev functions $f\in L^p(0,\tau;N^{1,p}(\Omega))$, $p\in[1,\infty)$. Following the previous notation, we will now have $A=\ztau$, $\tau>0$, $\nu=\leb^1$ - the Lebesgue measure on $\R$ - and $\Y=N^{1,p}(\Omega)$ with $\Omega$ open set in the metric measure space $(\X,d,\mu)$.

\medskip

\begin{definition}\label{par-sob} 
 We define the \textit{parabolic Newton-Sobolev space} $L^p(0,\tau; N^{1,p}(\Omega))$, $p\in[1,\infty)$, to be the vector space of all equivalence classes of measurable functions $f:\ztau\to N^{1,p}(\Omega)$ such that
 \begin{equation*}
  \int_0^\tau\|f(t)\|^p_{N^{1,p}(\Omega)}\d t<\infty.
 \end{equation*}
In other words, for $\leb^1$-almost every $t\in\ztau$, the map $t\mapsto f(t)$ defines a function in $N^{1,p}(\Omega)$. $L^p(0,\tau;N^{1,p}(\Omega))$ is a Banach space equipped with the norm
\begin{equation*}
 \|f\|_{L^p(0,\tau;N^{1,p}(\Omega))}\coloneqq\left(\int_0^\tau\|f(t)\|_{N^{1,p}(\Omega)}^p\d t\right)^\frac{1}{p}.
\end{equation*}
\end{definition}

By the discussion of Section \ref{sec-banach-lp} we clearly have that for any $f\in L^p(0,\tau;N^{1,p}(\Omega))$ there is a sequence of simple functions $(f_k)_{k\in\N}\subset N^{1,p}(\Omega)$,
\begin{equation}\label{simple-np}
 f_k=f_k(t)=\sum_{i=1}^{n_k}\one_{E_i^{(k)}}(t)\cdot v_i^{(k)}\quad\forall\,k\in\N
\end{equation}
where $\{E_i^{(k)}\}_{i=1}^{n_k}$ is a measurable disjoint partition of $\ztau$ and $\{v_i^{(k)}\}_{i=1}^{n_k}\subset N^{1,p}(\Omega)$, such that $f$ is the $\leb^1$-almost everywhere limit on $\ztau$ of the $f_k$'s, or that (upon relabeling)
\begin{equation}\label{ae-lim-np}
 f(t)=\sum_{k=1}^\infty\one_{E_k}(t)\cdot v_k
\end{equation}
with the $E_k$'s and the $v_k$'s to be intended as in \eqref{simple-np}. As already noted right after Definition \ref{def-ban-lp}, the pointwise $\leb^1$-almost everywhere convergence \eqref{ae-lim-np} means that
\begin{equation*}\label{simple-approx}
 \|f-f_k\|_{N^{1,p}(\Omega)}\longrightarrow0\qquad\text{as}\;k\to\infty\;\text{for}\:\leb^1\text{-almost\:every}\;t\in\ztau.
\end{equation*}
Of course, Bochner's Theorem applies also in this case so we have that any parabolic function  $f\in L^p(0,\tau;N^{1,p}(\Omega))$ is integrable in the sense of Bochner.

\medskip

Together with the definition of time-dependent Newton-Sobolev spaces, it comes a natural notion of the parabolic minimal $p$-weak upper gradients:

\medskip

\begin{definition}
Let $f\in L^p(0,\tau;N^{1,p}(\Omega))$. We define the \textit{parabolic minimal $p$-weak upper gradient} of $f$ by simply setting
\begin{equation}\label{par-up-grad}
 g_f=g_{f(t)}
\end{equation}
for $\leb^1$-almost every $t\in\ztau$.
\end{definition}

We observe that the above notion is well posed since for $\leb^1$-almost every $t\in\ztau$, $t\mapsto f(t)\in N^{1,p}(\Omega)$ and then \eqref{par-up-grad} is well defined.

\bigskip
\sloppy
It is worth to notice that in the characterization of the parabolic spaces $L^p(0,\tau;N^{1,p}(\Omega))$ it is intrinsically implied a \textit{time-slice} approach, as the functions together with their parabolic minimal $p$-weak upper gradients are defined up to $\leb^1$-negligible subsets of $\ztau$.

\smallskip

It is also important to spend some words on the $\mu\otimes\leb^1$-measurability in the product space $\omt$ for $f\in L^p(0,\tau;N^{1,p}(\Omega))$ and for its parabolic minimal $p$-weak upper gradient $g_{f(t)}$. Indeed, our main focus is on time-smoothing and, since the implementation of this technique in the proof of regularity or existence results involves partial integration with respect to the time variable, one has to ensure that it is possible to apply Fubini's Theorem.

\medskip

\begin{remark}[Measurability in the product space $\omt$]  \label{prod-meas} Since by definition $N^{1,p}(\Omega)\subset L^p(\Omega)$, the map $t\mapsto f(t)\in N^{1,p}(\Omega)$ which defines a parabolic Sobolev function $f\in L^p(0,\tau;N^{1,p}(\Omega))$ is clearly also a map from $\ztau$ to $L^p(\Omega)$ for $\leb^1$-almost every $t\in\ztau$, and of course
\begin{equation*}
 \int_0^\tau \|f(t)\|^p_{L^p(\Omega)}\d t<\infty,
\end{equation*}
so we have actually a map in $L^p(0,\tau;L^p(\Omega))$, as it is obvious to expect since $L^p(0,\tau;N^{1,p}(\Omega))$ $\subset L^p(0,\tau;L^p(\Omega))$ by the very definitions of the spaces.
Now, as
\begin{equation*}
L^p(0,\tau;L^p(\Omega))=L^p(\omt),
 \end{equation*}
 (see for instance \cite[Section 2.1.1]{ruz}) we get that all the functions in $L^p(0,\tau;N^{1,p}(\Omega))$ are also measurable in the product space $\omt=\Omega\times\ztau$ equipped with the product measure $\mu\otimes\leb^1$.
 
 Let us now turn to parabolic minimal $p$-upper gradients. It is clear that since $f\in L^p(0,\tau;N^{1,p}(\Omega))$, then one has $g_f=g_{f(t)}\in L^p(\Omega)$ for $\leb^1$-almost every $t\in\ztau$ and
 \begin{equation*}
  \int_0^\tau \|g_{f(t)}\|^p_{L^p(\Omega)}\d t<\infty,
 \end{equation*}
so that, in other words, $t\mapsto g_{f(t)}$ defines a map in $L^p(0,\tau;L^p(\Omega))=L^p(\omt)$, from which we infer that $g_{f(t)}$ is $\mu\otimes\leb^1$-measurable in $\omt$ as well.

In particular, $g_{f(t)}$ is also strongly measurable in the sense of $L^p(0,\tau;L^p(\Omega))$. Indeed, as $f(t)$ is the $\leb^1$-almost everywhere limit of the $f_k$'s defined in \eqref{simple-np}, being strongly measurable in the sense of $L^p(0,\tau;N^{1,p}(\Omega))$, combining \eqref{ae-lim-np} with the locality of minimal $p$-weak upper gradients we get that
\begin{align}\label{strong-up-grad}\begin{split}
 g_{f(t)} & = g_{\sum_{k=1}^\infty \one_{E_k}(t)\cdot v_k} \\ & = \sum_{k=1}^\infty \one_{E_k}(t)\cdot g_{v_k}\end{split}
\end{align}
$\mu\otimes\leb^1$-almost everywhere since the $E_k$'s are taken to be disjoint and the functions $\one_{E_k}(t)$ act as scalars with respect to upper gradients. Obviously, as $\{v_k\}_{k=1}^\infty\subset N^{1,p}(\Omega)$, $g_{v_k}\in L^p(\Omega)$ for all $k\in\N$, so the strong measurability follows naturally. Also, all of the above is equivalent to say that $g_{f(t)}$ is approximated in $L^p(\Omega)$ by the sequence $(g_{f_k})_{k\in\N}\subset L^p(\Omega)$,
\begin{equation*}
 g_{f_k(t)} = g_{\sum_{i=1}^{n_k} \one_{E_i^{(k)}}(t)\cdot v_i^{(k)}} = \sum_{i=1}^{n_k} \one_{E_i^{(k)}}(t)\cdot g_{v_i^{(k)}}
\end{equation*}
which we obtain from \eqref{simple-np} by arguing as in \eqref{strong-up-grad}, with the usual understanding of the $E_i^{(k)}$'s and of the $v_i^{(k)}$'s.
\end{remark}

\smallskip

\begin{remark} We note that by the discussion above it turns out that the functions $f=f(t)$ in the parabolic space $L^p(0,\tau;N^{1,p}(\Omega))$, which depend formally on the variable $t$, can actually be treated, in practice, as $\mu\otimes\leb^1$-measurable functions of both $x$ and $t$, i.e. as measurable functions $f:\omt\to\R$. Of course, the same point of view and interpretation can be applied to the parabolic minimal $p$-weak upper gradients as well.
\end{remark}


\subsection{Time-smoothing}\label{sec-smoothing}

\bigskip

Now that we have discussed the main definition and properties of parabolic Newton-Sobolev spaces and we have a consistent characterization of parabolic minimal $p$-weak upper gradients, we can proceed towards our result on time-smoothing. To get things started, we first recall the notion of time mollification:

\bigskip

\begin{definition}\label{t-moll}
 Given $f\in \lploc(0,\tau; N^{1,p}_\loc(\Omega))$ and a standard mollifier \[\eta_\varepsilon(s)=\dfrac{1}{s}\eta\left(\dfrac{s}{\eps}\right),\]
 $\eps>0$, we define the \textit{time mollification} $f_\eps$ of $f$ as follows:
\begin{equation}\label{eq-t-moll}
 f_\eps(t)\coloneqq\int_{-\eps}^{\eps}\eta_\eps(s)f(t-s)\d s.
\end{equation}
\end{definition}

\medskip

\begin{theorem}\label{vbthm}
 Let $(\X,d,\mu)$ be a complete and separable metric measure space equipped with a non-negative Radon measure $\mu$, and let $\Omega\subset\X$ be an open set. Then, for any $f\in L^p_\loc(0,\tau; N^{1,p}_\loc(\Omega))$, $p\in[1,\infty)$, if we denote by $f_\eps$ the time mollification of $f$, $\eps>0$, we have $g_{f_\eps-f}\to0$ in $L^p_\loc(\omt)$ as $\eps\to0$. Moreover, as $s\to0$, we have $g_{f(t-s)-f(t)}\to0$ in $L^p_\loc(\omt)$ uniformly in $t$.
\end{theorem}

\begin{proof} Since we have
\[
 f(t)=\sum_{k=1}^\infty \one_{E_k}(t)\cdot v_k
\]
for $\leb^1$-almost every $t\in\ztau$, applying the definition \eqref{eq-t-moll} of time mollification gives
\begin{align*}
 f_\eps(t)=\int_{-\eps}^\eps \eta_\eps(s)f(t-s)\d s&=\int_{-\eps}^\eps \eta_\eps(s)\sum_{k=1}^\infty \one_{E_k}(t-s)\cdot v_k \,\d s \\
 & = \sum_{k=1}^\infty \left(\int_{-\eps}^\eps \eta_\eps(s)\one_{E_k}(t-s)\,\d s\right)\cdot v_k \\ & = \sum_{k=1}^\infty (\one_{E_k})_\eps(t)\cdot v_k, 
\end{align*}
which immediately yields
\[
 f(t)-f_{\eps}(t)=\sum_{k=1}^\infty \left(\one_{E_k}-(\one_{E_k})_\eps\right)\cdot v_k
\]
$\leb^1$-almost everywhere on $\in\ztau$.

Now, by the locality and by the sub-linearity of minimal $p$-weak upper gradients, we easily infer that
\begin{align*}
 g_{f-f_\eps} & = g_{\sum_{k=1}^\infty \left(\one_{E_k}-(\one_{E_k})_\eps\right)\cdot v_k} \\ & \le \sum_{k=1}^\infty g_{\left(\one_{E_k}-(\one_{E_k})_\eps\right)\cdot v_k} \\ & = \sum_{k=1}^\infty \left\vert\one_{E_k}-(\one_{E_k})_\eps\right\vert\cdot g_{v_k}\underset{\eps\to0}{\longrightarrow}0
\end{align*}
$\leb^1$-almost everywhere on $\ztau$ by the standard properties of mollifications, meaning also  that $g_{f-f_\eps}\to0$ as $\eps\to0$ $\mu\otimes\leb^1$-almost everywhere in $\omt$, which eventually forces $g_{f-f_\eps}\to0$ in $\lploc(\omt)$ as $\eps\to0$.

We explicitly observe again that the factors $\one_{E_k}-(\one_{E_k})_\eps$, being functions of $t$ only, act scalarly with respect to upper gradients.

\medskip

It remains to prove that $g_{f(t-s)-f(t)}\to0$ in $\lploc(\omt)$ as $s\to0$, uniformly in $t$. Again by the locality of the minimal $p$-weak upper gradient, and by the observations in Remark \ref{prod-meas}, for every $s>0$ one has
\[
 g_{f(t-s)-f(t)}=\sum_{k=1}^\infty\left\vert\one_{E_k}(t-s)-\one_{E_k}(t)\right\vert\cdot g_{v_k}
\]
$\mu\otimes\leb^1$-almost everywhere in $\omt$. Then, we can consider $K\Subset\omt$ of the form $K=U\times I$ with $U\Subset\Omega$ and $I$ compactly contained in $\ztau$, to find
\sloppy
\begin{align*}
\|g_{f(t-s)-f(t)}\|_{L^p(K)} & = \left(\int_K g^p_{f(t-s)-f(t)}\d\mu\d t\right)^{\frac{1}{p}}\\ & =\left(\int_K\left(\sum_{k=1}^\infty\left\vert\one_{E_k}(t-s)-\one_{E_k}(t)\right\vert\cdot g_{v_k}\right)^p\d\mu\d t\right)^{\frac{1}{p}}\\
& =\left\Vert\sum_{k=1}^\infty\left(\one_{E_k}(t-s)-\one_{E_k}(t)\right)\cdot g_{v_k}\right\Vert_{L^p(K)}\\
& \le \sum_{k=1}^\infty\left\Vert\left(\one_{E_k}(t-s)-\one_{E_k}(t)\right)\cdot g_{v_k}\right\Vert_{L^p(K)}\\
&= \sum_{k=1}^\infty\left(\int_K\left\vert\one_{E_k}(t-s)-\one_{E_k}(t)\right\vert^p\cdot g_{v_k}^p\d\mu\d t\right)^{\frac{1}{p}}\\
& =\sum_{k=1}^\infty\left(\int_I\left\vert\one_{E_k}(t-s)-\one_{E_k}(t)\right\vert^p\d t\cdot\int_U g_{v_k}^p\d\mu\right)^{\frac{1}{p}}\\
& = \sum_{k=1}^\infty\|\one_{E_k}(t-s)-\one_{E_k}(t)\|_{L^p(I)}\cdot\|g_{v_k}\|_{L^p(U)},
\end{align*}

where we used Minkowski's Inequality and Fubini's Theorem.\\Now, since the functions $\one_{E_k}\in L^p_\loc\ztau$, the expression above vanishes uniformly in $t$ as $s\to0$ by the continuity of translations on $L^p$ functions.

\end{proof}

\bigskip

\begin{adjustwidth}{1cm}{1cm}
\begin{center}\textbf{Aknowledgements.}\end{center} 
This research was partially supported by the Academy of Finland during the Author's stay at the Department of Mathematics and Systems Analysis at Aalto University (Espoo, Finland).\\
We wish to thank Juha K. Kinnunen for valuable discussions on the topic of time-smoothing in metric measure spaces and for encouraging the writing of this paper.
\end{adjustwidth}

\bigskip


\end{document}